\documentclass[draft]{svjour3}
\smartqed

\usepackage{amsmath,amstext,amssymb, amsfonts, mathrsfs}
\usepackage[dvips]{graphics,color}
\usepackage{mathptmx}
\newcommand{\B}{\mathfrak{B}}

\newcommand{\C}{\mathcal}
\newcommand{\BB}{\mathbb}

\spnewtheorem{assumption}{Assumption}{\bf}{\rm}
\spnewtheorem*{proofx}{}{\it}{\rm}

\journalname{}

\allowdisplaybreaks

\begin{document}

\title{Kolmogorov's Equations for Jump Markov Processes with Unbounded Jump Rates \thanks{The first author was partially supported by the National Science Foundation [Grants CMMI-1335296 and   CMMI-1636193.].}}

\author{Eugene Feinberg \and Manasa Mandava \and Albert N. Shiryaev}

\institute{E. Feinberg
			\at Department of Applied Mathematics and Statistics, Stony Brook University,  NY 11794-3600, USA \\
              Tel.: +1-631-632-7189, Fax: +1-631-632-8490\\
              \email{Eugene.Feinberg@stonybrook.edu}           
           \and
           M. Mandava
           \at Indian School of Business, Hyderabad 500032, India
           \and
           A. N. Shiryaev
           \at Steklov Mathematical Institute, 8, Gubkina Str., Moscow 119991, Russia
}

\date{Received: date / Accepted: date}

\titlerunning{Kolmogorov's equations for jump Markov processes}
\authorrunning{Feinberg et al.}

\maketitle

\begin{abstract}
As is well-known, transition probabilities of jump Markov processes satisfy Kolmogorov's backward and forward equations.  In the seminal 1940 paper, William Feller investigated solutions of Kolmogorov's equations for jump Markov processes.  Recently the authors solved the problem studied by Feller and showed that the minimal solution of Kolmogorov's backward and forward  equations is the transition probability of the corresponding jump Markov process if the transition rate at each state is bounded.  This paper presents more general results.  For  Kolmogorov's backward equation, the sufficient condition for the described property of the minimal solution is that the transition rate  at each state is locally integrable, and for Kolmogorov's forward equation the corresponding sufficient condition is that the transition rate at each state is locally bounded.
\end{abstract}


\section{Introduction}
Continuous-time jump Markov processes are broadly used in stochastic models of operations research. In many applications continuous-time jump Markov processes are defined by transition rates often called $Q$-functions.   Each $Q$-function defines Kolmogorov's backward and forward equations,  and transition  probabilities of the  jump Markov process defined by  the $Q$-function  satisfy these equations. 
If transition rates are unbounded,  Kolmogorov's equations may have multiple solutions, see, e.g., Anderson~\cite[Chap.~4, Example~1.2]{Anderson}, Doob~\cite[Chap. 6]{Doob}, Kendall~\cite{Ken}, Reuter~\cite{Reu},  and the relation between Kolmogorov's equations and the corresponding transition probabilities is not trivial. 
For example, in queueing theory birth and death processes have unbounded transition rates in each of the following three situations:  arrival rates depend on the state of the system and are unbounded, queues with an infinite number of servers, queues with reneging.

This paper answers the  questions on how a nonhomogeneous jump Markov process can be defined for a given  $Q$-function and how can its transition probability be found as a solution of Kolmogorov's backward and forward equations. These questions were studied by Feller~\cite{Fel} for continuous $Q$-functions and a standard Borel state space, by Ye et al.~\cite{YGHL}  for measurable $Q$-functions and a countable state space, and by Feinberg et al.~\cite{FMS} for measurable $Q$-functions and a standard Borel state space.  All these papers considered  $Q$-functions satisfying certain boundedness conditions. This paper generalizes the results from Feinberg et al.~\cite{FMS} to more general classes of unbounded $Q$-functions, strengthens some of results from \cite{FMS}, and provides proofs of the following two facts:  (i) (Lemma~\ref{l:A-eq}(a)) Fellers's assumption on the boundedness of a $Q$-function, Assumption~\ref{Feller}, is equivalent to the boundedness of a $Q$-function at each state, Assumption~\ref{LB}, and (ii) (Theorem~\ref{thm:intFKE}) Kolmogorov's forward equation is equivalent to the integral  equation~\eqref{int-FKE}.  The  first fact is  introduced  and the validity of equation~\eqref{int-FKE} is stated in \cite{FMS} without detailed proofs.  

For a topological space $S,$ its Borel $\sigma$-field (the $\sigma$-field generated by  open subsets of $S$) is always denoted by $\B(S),$ and the sets in $\B(S)$ are called {\it Borel subsets} of $S$. Let $\BB{R}$ be the real line endowed with the Euclidean metric.   A topological space $(S, \B(S))$ is called a {\it standard Borel space} if there exists a bijection $f$ from $(S, \B(S))$ to a Borel subset of $\BB{R}$ such that the mappings $f$ and $f^{-1}$ are measurable. In this paper,
measurability and Borel measurability are used synonymously.
 Let $({\bf X}, \B({\bf X}))$ be a standard Borel space, called the state space, and let $[T_0, T_1[$ be a finite or an infinite interval in $\BB{R}_+:= [0, \infty[$.  In this paper, we always assume that $T_0< T_1.$
A function $P(u, x; t, B)$, where $u \in [T_0, T_1[$, $t \in ]u, T_1[$, $x \in {\bf X},$ and $B \in \B({\bf X}),$ is called  a
transition function if it takes values in $[0,1]$
and satisfies the following properties:
\begin{itemize}
\item[(i)] for all $u,x,t$ the function $P(u,x;t,\cdot)$ is a
measure on $({\bf X}, \mathfrak{B}({\bf X}))$;
\item[(ii)]for all $B$ the
function $P(u,x; t, B)$ is Borel measurable in $(u,x,t);$
\item[(iii)]
$P(u,x;t,B)$ satisfies the Chapman-Kolmogorov equation
\begin{equation}
\label{CKE} P(u, x; t, B) = \int_{{\bf X}} P(s,y; t, B)P(u, x; s, dy),
\qquad u < s < t.
\end{equation}
\end{itemize}
A transition function $P$ is called  {\it
regular} if $P(u,x;t,{\bf X}) = 1$ for all $u,x,t$ in the domain of $P$.

A stochastic process $\{\BB{X}_t: t \in [T_0, T_1[\}$ with values in ${\bf X}$,
defined on the probability space $(\Omega, \C{F}, \BB{P})$ and adapted to the
filtration $\{\C{F}_t\}_{t \in [T_0, T_1[}$, is called Markov if
$\BB{P}(\BB{X}_t \in B \mid \C{F}_u) = \BB{P}(\BB{X}_t \in B \mid
\BB{X}_u)$, $\BB{P}-a.s.$ for all $u \in [T_0, T_1[,$ $t \in ]u,T_1[$, and  $B
\in \B({\bf X})$. Each Markov process has a transition function $P$ such
that $\BB{P}(\BB{X}_t \in B \mid \BB{X}_u) = P(u,\BB{X}_u; t, B),$
$\BB{P}-a.s.$; 
see Kuznetsov~\cite{Kuz}, where the equivalence of two
definitions of a Markov process given by Kolmogorov~\cite{Kol}
is established. In addition, if a Markov process is a jump process, that is, if each sample path of the process is a right-continuous piecewise constant function in $t$ that has a finite or countable number of discontinuity points on $t \in [T_0, T_1[$, then the Markov
process is called a {\it jump Markov process}.

A function $q(x,t,B)$, where $x \in {\bf X}$, $t \in [T_0, T_1[$, and $B \in \B({\bf X})$,
is called a {\it Q-function} if it satisfies the following
properties:
\begin{itemize}
\item[(a)]for all $x,t$ the function $q(x,t,\cdot)$ is a signed
measure on $({\bf X}, \mathfrak{B}({\bf X}))$ such that $q(x,t,{\bf X})$ $\leq$ $0$ and
$0 \leq q(x, t, B \setminus \{x\}) < \infty$ for all $B \in
\B({\bf X})$;
\item[(b)] for all $B$ the
function $q(x,t,B)$ is measurable in $(x,t).$
\end{itemize}
In addition to properties (a) and (b), if $q(x,t,{\bf X}) =0$ for all
$x,t$, then the $Q$-function $q$ is  called {\it conservative}.
Note that any $Q$-function can be transformed into a conservative
$Q$-function  by adding an absorbing state $\bar{x}$ to ${\bf X}$ with
$q(x,t,\{\bar{x}\}):= -q(x,t,{\bf X})$, $q(\bar{x}, t, {\bf X}) := 0$, and
$q(\bar{x}, t, \{\bar{x}\}) := 0$, where $x \in {\bf X}$ and $t \in
[T_0, T_1[$. To simplify the presentation, in this paper we
always assume that $q$ is conservative. The same arguments as in Remark~4.1 in Feinberg et al.~\cite{FMS} explain how the main formulations change when the $Q$-function $q$ is not conservative. A $Q$-function $q$ is called {\it continuous}
if it is continuous in $t \in [T_0, T_1[$.

Feller~\cite{Fel} studied Kolmogorov's backward and forward equations for continuous
$Q$-functions and provided explicit formulae for a
transition function that satisfies Kolmogorov's backward and forward equations. If the constructed transition function is
regular, Feller~\cite[Theorem 3]{Fel} showed that this transition
function is the unique solution of Kolmogorov's backward equation. Though Feller~\cite{Fel} focused on regular
transition functions,  it follows from the proof of Theorem 3 in
Feller~\cite{Fel} that the  transition function constructed there
is the minimal solution of Kolmogorov's backward equation. Feinberg et al.~\cite{FMS} showed  for a measurable $Q$-function that the transition function constructed by Feller~\cite{Fel} is the minimal solution of Kolmogorov's backward and forward equations, and it is the transition function of the jump Markov process defined by the random measure whose compensator is defined via the $Q$-function. In this paper, we show that the  minimal solution of Kolmogorov's backward and forward  equations is the transition function of the corresponding jump Markov process under more general boundedness assumptions on $Q$-functions than those assumed in  \cite{FMS}. 

\section{Assumptions and description of  main results}
\label{S-A}
In this section, we describe several assumptions on unbounded $Q$-functions and the results of this paper.  Let $q(x,t) := -q(x, t, \{x\})$ for $x \in {\bf X}$ and $t \in [T_0, T_1[,$ and let ${\bar q}(x):= \sup_{t \in [T_0, T_1[} q(x,t)$ for $x\in {\bf X}.$
 Feller~\cite{Fel} studied Kolmogorov's equations for continuous $Q$-functions under the following assumption.
\begin{assumption}[{\it Feller's assumption}]
\label{Feller}
There exists Borel subsets $B_n, n = 1,2,\ldots,$ of ${\bf X}$ such that $\sup_{x \in B_n} {\bar q}(x) < n$ for all $n = 1,2,\ldots$  and $B_n \uparrow {\bf X}$ as $n \to \infty$.
\end{assumption}

Feinberg et al.~\cite{FMS} studied Kolmogorov's equations for measurable $Q$-functions under the following assumption with $T_0 = 0$ and $T_1 = \infty$.
\begin{assumption}[{\it Boundedness of $q$}]
\label{LB}
${\bar q}(x) < \infty $ for each $x \in {\bf X}$.
\end{assumption}

As mentioned in Feinberg et al.~\cite[p.~262]{FMS}, Assumptions~\ref{Feller} and \ref{LB} are equivalent; see Lemma~\ref{l:A-eq}(a) for details. In this section, we introduce two more general assumptions.

\begin{assumption}[{\it Local boundedness of $q$}]
\label{ALB}
$\sup_{t\in [T_0,s[} q(x,t)<\infty$ for each $s\in ]T_0,T_1[$ and  $x\in {\bf X}.$
\end{assumption}

\begin{assumption}[{\it Local $\mathcal{L}^1$ boundedness of $q$}]
\label{L1}
$\int_{T_0}^s q(x,t) dt < \infty$  for each $s \in ]T_0, T_1[$ and  $x \in {\bf X}.$
\end{assumption}

The following lemma compares Assumptions~\ref{Feller}--\ref{L1}.

\begin{lemma}
\label{l:A-eq}
The following statements hold for a measurable $Q$-function $q:$

(a) Assumptions~\ref{Feller} and \ref{LB} are equivalent;

(b) Assumption~\ref{LB} implies Assumption~\ref{ALB};

(c) Assumption~\ref{ALB} implies Assumption~\ref{L1}.
\end{lemma}
\begin{proof}
(a) Let $\{B_n, n = 1,2,\ldots\}$ be a sequence of Borel subsets of ${\bf X}$ satisfying the properties stated in Assumption~\ref{Feller}. Then for each $x \in {\bf X}$ there exists an $n \in \{1,2,\ldots\}$ such that $x \in B_n$ and therefore ${\bar  q}(x) < n$. Thus, Assumption~\ref{Feller} implies Assumption~\ref{LB}. To prove that Assumption~\ref{LB} implies Assumption~\ref{Feller}, define $C_n := \{x \in {\bf X}:  {\bar q}(x) \ge n\}, n = 1,2,\ldots\ .$  Since $C_n =\cap_{k=1}^\infty proj_{\bf X} (\{(x,t) \in ({\bf X} \times [T_0, T_1[) : q(x,t) \ge n-k^{-1}\})$ are countable intersections of projections of Borel sets, the sets $C_n$ are analytic for all $n = 1,2,\ldots;$ see Bertsekas and Shreve~\cite[Proposition~7.39 and Corollary~7.35.2]{BS}. In addition, Assumption~\ref{LB} implies  that $\bigcap_{n = 1}^\infty C_n = \emptyset.$ Thus, in view of the Novikov separation theorem, Kechris~\cite[Theorem~28.5]{Kechris}, there exist Borel subsets $Z_n,$ $n=1,2,\ldots,$ of ${\bf X}$ such that $C_n \subseteq Z_n$  and $\bigcap_{n = 1}^\infty Z_n = \emptyset.$ This fact implies that $Z_n^c \subseteq C_n^c$ and $\bigcup_{n = 1}^\infty Z_n^c = {\bf X},$ where the sets $Z_n^c$ and $C_n^c$ are complements of the sets $Z_n$ and $C_n$, respectively.  Let $B_n:= \cup_{m = 1}^n Z_m^c$ for all $n = 1,2,\ldots\,.$  The Borel sets $B_n,$ $n=1,2,\ldots,$ satisfy the properties stated in Assumption~\ref{Feller}.

(b,c)  Statements (b) and (c) are obvious. \qed
\end{proof}

\begin{remark}
\label{T-eq}
Under Assumption~\ref{Feller}, which, as stated in Lemma~\ref{l:A-eq}(a), is equivalent to Assumption~\ref{LB}, Feller~\cite{Fel} studied Kolmogorov's equations for the time parameter  $t\in [T_0, T_1[.$  Under Assumption~\ref{LB}, Feinberg et al.~\cite{FMS} studied Kolmogorov's equations for the time parameter $t\in [T_0, T_1[ = [0,\infty[.$ 
It is apparent that the formulation  of results for an arbitrary interval $[T_0, T_1[,$  where $0\le T_0< T_1\le\infty,$ is more general than 	their formulation for the interval $[0,\infty[.$  In fact, these two formulations are equivalent under Assumption~\ref{LB} holding for the corresponding time intervals.  Indeed, a $Q$-function $q,$ defined for $t\in [T_0,T_1[$ and satisfying Assumption~\ref{LB} on this interval, can be extended to
all $t\in [0,\infty[$ by setting $q(x,t,B):=0$ for $x\in {\bf X},$ $t\in [0, T_0[\cup [T_1,\infty[,$  and $B\in\mathfrak{B}({\bf X}).$ The extended $Q$-function satisfies  Assumption~\ref{LB} for $t\in [0,\infty[.$   Since solutions of  Kolmogorov's equations \eqref{eq:BKDE} and \eqref{eq:FKDE} for the extended $Q$-function are constants in $t,$ when $t\in [0,T_0[$ and  $t\in [T_1,\infty[,$ and since Kolmogorov's equations for the original $Q$-function $q$ and the extended $Q$-function coincide when $t \in [T_0, T_1[,$ there is a one-to-one correspondence between solutions of Kolmogorov's equations for the $Q$-function $q$ and for the extended $Q$-function. Since Assumption~\ref{LB} is assumed in \cite{FMS}, the results obtained in~\cite{FMS} for the problem formulations for the interval  $[0,\infty[$ hold for an arbitrary  interval $[T_0,T_1[.$ 
\end{remark}

As explained in Remark~\ref{T-eq}, if a $Q$-function $q$ satisfies  Assumption~\ref{LB} on the interval $[T_0, T_1[,$ its extension to $t\in [0,\infty[$ defined in  Remark~\ref{T-eq}  satisfies the same assumption on $[0, \infty[.$  The following example illustrates that this need not be the case if $q$ satisfies Assumption~\ref{ALB} or Assumption~\ref{L1}. Hence, following Feller~\cite{Fel}, we formulate the results in this paper for an arbitrary interval $[T_0, T_1[$ with $0\le T_0 < T_1<\infty.$
\begin{example}
\label{ex:B-eq}
{\it A $Q$-function $q$ satisfies  Assumption~\ref{ALB} on the interval $[T_0, T_1[,$ while its extention to $t\in [0,\infty[$  defined in Remark~\ref{T-eq} does not satisfy even the weaker  Assumption~\ref{L1} when $T_0 = 0$ and $T_1 = \infty$. }Fix an arbitrary $T_1 \in ]0, \infty[.$ Let $T_0 := 0$, ${\bf X} := \{1,2,\ldots\},$ and $q(x,t):= \frac{1}{T_1-t}$ for all $x \in {\bf X}, 0 \le t < T_1$. Then $\sup_{t\in [T_0,s[} q(x,t)\le (T_1-s)^{-1}<\infty$ for each $s\in ]T_0,T_1[$ and  $x\in {\bf X}.$ Thus the $Q$-function $q$ satisfies Assumption~\ref{ALB}.

Consider the extension of $q$ to $t\in [0,\infty[$ defined in Remark~\ref{T-eq}  and the sequence $\{t_m, m = 1,2,\ldots\}  \subset [0, T_1[$ with $t_m = T_1 - \frac{1}{m}$ for all $m = 1,2,\ldots\ .$ Observe that 
\begin{equation}
\int_{0}^{T_1} q(x,s)ds = \lim_{m \to \infty}\int_{0}^{t_m} q(x,s)ds = \lim_{m \to \infty} \int_{0}^{t_m} \frac{1}{T_1-s}ds = \lim_{m \to \infty} \log{(m \times T_1)} = \infty.
\end{equation}
Therefore, the described extension of $q$ from $t\in [0,T_1[$ to  $t\in [0,\infty[$ does not satisfy Assumption~\ref{L1}. 
\qed
\end{example}

In Section~\ref{S-BE} we show in Theorem~\ref{thm:JMP} that under Assumption~\ref{L1} the compensator  defined by a $Q$-function and an initial probability measure define a jump Markov process, whose transition function $\bar P$ is described in \eqref{def}, and Theorem~\ref{thm:BKE} states that this function is the minimal function satisfying  Kolmogorov's backward equation.  The function $\bar P$ was introduced in Feller~\cite{Fel}.  Section~\ref{S-FE} deals with Kolmogorov's forward equation, when Assumption~\ref{ALB} holds, and Theorem~\ref{thm:FKDE} states that $\bar P$ is  the minimal function satisfying the forward equation.    Section~\ref{S-GB} presents results on Kolmogorov's forward  equation under Assumption~\ref{LB}.

\section{Jump Markov process defined by a $Q$-function and Kolmogorov's backward equation}
\label{S-BE}

In this section, we show that a $Q$-function satisfying Assumption~\ref{L1} defines a transition function for a jump Markov process.  In addition, this transition function is the minimal  function satisfying Kolmogorov's backward equation defined by this $Q$-function.

Let $x_\infty\notin {\bf X}$ be an isolated point adjoined to the space
${\bf X}$. Denote ${\bar {\bf X}}={\bf X}\cup\{x_\infty\}$.
Consider the Borel $\sigma$-field $\B({\bar {\bf X}})=\sigma(\mathfrak{B}({\bf X}),\{x_\infty\})$ on $\bar {\bf X}$, which is
the minimal $\sigma$-field containing $ \mathfrak{B}({\bf X})$ and
$\{x_\infty\}.$ Let $({\bar {\bf X}} \times ]T_0, T_1])^\infty$ be the set of all sequences $(x_0, t_1, x_1,
t_2, x_2, \ldots)$ with $x_n\in \bar{{\bf X}}$ and $t_{n+1}\in
]T_0, T_1]$ for all $n =0,1,\ldots\ .$ This set is endowed
with the $\sigma$-field generated by the products of the Borel
	$\sigma$-fields $\B(\bar{{\bf X}})$ and $\B(]T_0, T_1])$.

Denote by $\Omega$ the subset of all sequences $\omega= (x_0, t_1,
x_1, t_2, x_2, \ldots)$ from $({\bar {\bf X}} \times ]T_0, T_1])^\infty$ such that: (i) $x_0 \in {\bf X}$; (ii) for all $n =1,2,\ldots\,,$ if $t_n < T_1$, then $t_n < t_{n+1}$ and
$x_n \in {\bf X}$, and if $t_n = T_1$, then $t_{n+1} = t_n$ and
$x_n = x_\infty$. Observe that $\Omega$ is a measurable subset of
$({\bar {\bf X}} \times ]T_0, T_1])^\infty$. Consider the measurable space $(\Omega,
\C{F})$, where $\C{F}$ is the $\sigma$-field of the measurable subsets
of $\Omega$. For all $n = 0,1,\ldots$, let $x_n(\omega)=x_n$ and
$t_{n+1}(\omega)=t_{n+1},$ where $\omega \in \Omega,$ be the random variables defined on the
measurable space $(\Omega, \C{F})$. Let $t_0 := T_0$, $t_\infty
(\omega) := \lim\limits_{n \to \infty} t_n (\omega)$, $\omega \in \Omega$, and for all $t \in [T_0, T_1],$ let $\C{F}_t := \sigma(\B({\bf X}), \C{G}_t)$, where $\C{G}_t := \sigma (I\{x_n \in B\}I\{t_n \le s\}: n \ge 1, T_0 \le s \le t, B \in \B({\bf X})).$ Throughout this paper, we omit $\omega$ whenever possible. 

Consider the multivariate point process $(t_n, x_n)_{n =1,2,\ldots}$ on $(\Omega, \C{F})$. Given a $Q$-function $q$ satisfying Assumption~\ref{L1}, define a random measure $\nu$ on $([T_0, T_1[ \times {\bf X})$ as
\begin{equation}
\label{compensator}
\nu(\omega; [T_0,t], B): = \int_{T_0}^{t}\sum_{n \ge 0}I\{t_n < s \leq t_{n+1}\} q(x_n,s, B \setminus \{x_n\})ds,  \quad  t \in [T_0, T_1[,\ B \in \B({\bf X}).
\end{equation}
Observe that $\nu$ is a predictable random measure. Indeed, formula~(\ref{compensator}) coincides with Feinberg et al.~\cite[Eq.~(2)]{FMS} when $T_0 = 0$ and $T_1 = \infty$. Arguments similar to those following  Feinberg et al.~\cite[Eq.~(2)]{FMS}, which show that the random measure $\nu$  defined in \cite[Eq.~(2)]{FMS} is a predictable random measure, imply that the measure $\nu$ defined in \eqref{compensator} is a predictable random measure. Furthermore, $\nu(\{t\}\times {\bf X}) \le 1$ for all $t\in ]T_0,T_1[$ and $\nu([t_\infty,\infty[\times {\bf X})=0. $  According to Jacod~\cite[Theorem~3.6]{Jac}, the predictable random measure $\nu$ defined in \eqref{compensator} and a  probability measure $\gamma$ on $\bf X$ define a unique probability measure $\BB{P}$ on $(\Omega, \C{F})$ such that $\BB{P}(x_0 \in B) = \gamma(B), B \in \B({\bf X}),$ and $\nu$ is the compensator of the random measure of the multivariate point process $(t_n, x_n)_{n \ge 1}$ defined by the triplet $(\Omega, \C{F}, \BB{P})$.

Consider the process $\{\BB{X}_t: t\in [T_0, T_1[\}$,
\begin{equation}
\label{jump}
\mathbb{X}_t(\omega) := \sum_{n \geq 0}{\bf I}\{t_n \leq t < t_{n+1}\}x_n  + {\bf I}\{t_\infty \leq t\}x_\infty,
\end{equation}
defined on $(\Omega, \C{F}, \BB{P})$ and adapted to the filtration $\{\C{F}_{t}, t \in [T_0, T_1[\}$. By definition, the process $\{\BB{X}_t:	 t \in [T_0, T_1[\}$ is a jump process.

For $x \in {\bf X}$ and $t \in [T_0, T_1[,$ let  $q^+(x,t,\cdot)$ be the measure on $({\bf X},\B({\bf X}))$ with values $q^+(x,t,B) := q(x,t,B\setminus \{x\}),$    $B \in \B({\bf X})$. 
In this paper, we use the notation
 \[q(x,t,dz\setminus\{x\}):=q^+(x,t,dz).\]
Following Feller \cite[Theorem 2]{Fel}, for $x \in {\bf X}$, $u \in [T_0, T_1[$, $t \in ]u,T_1[$, and $B \in \B({\bf X})$, define
\begin{equation}
\label{b0} \bar{P}^{(0)} (u,x;t,B): = I\{x \in B\} e^{-\int_u^t
q(x, s) ds},
\end{equation}
and 
\begin{equation}
\label{bn}
\bar{P}^{(n)}(u, x; t, B): = \int_{u}^{t} \int_{{\bf X}  }
 e^{ -\int_{u}^{w} q(x,\theta) d\theta }  q(x,w,dy \setminus \{x\}) \bar{P}^{(n-1)}(w, y; t, B)
 dw, \  n=1,2,\ldots\ .
\end{equation}
Set
\begin{equation}
\label{def}
\bar{P}(u, x; t, B) := \sum\limits_{n=0}^{\infty} \bar{P}^{(n)}(u, x; t, B).
\end{equation}
According to Feller~\cite[(27) and Theorem 4]{Fel}, equation \eqref{bn} can be rewritten as  
\begin{equation}
\label{bn-alt}
\bar{P}^{(n)}(u,x;t,B) = \int\limits_u^t \int\limits_{\bf X} \int\limits_{B } e^{-\int_w^t q(y,\theta)d\theta} q(z,w,dy\setminus \{z\}) \bar{P}^{(n-1)}(u,x;w,dz)dw,\  n = 1,2,\ldots\ .
\end{equation}
Though Feller~\cite{Fel}  considered continuous $Q$-functions, the proof of \eqref{bn-alt} given in Feller~\cite[Theorem~4]{Fel} remains correct for measurable $Q$-functions. 

Observe that $\bar{P}$ is a transition function if the $Q$-function $q$ satisfies Assumption~\ref{L1}. For continuous $Q$-functions satisfying Assumption~\ref{Feller}, Feller \cite[Theorems 2, 5]{Fel} proved that: (a) for fixed $u,x,t$ the function $\bar{P}(u,x;t,\cdot)$ is a measure on $({\bf X},\B({\bf X}))$  such that $0 \le \bar{P}(u,x;t,\cdot) \le 1$, and (b) for all $u,x,t,B$ the function $\bar{P}(u,x;t,B)$  satisfies the Chapman-Kolmogorov equation \eqref{CKE}. The proofs remain correct for measurable $Q$-functions satisfying Assumption~\ref{L1}. The measurability of $\bar{P}(u,x;t,B)$ in $u,x,t$ for all $B \in \B({\bf X})$  follows directly from the definitions \eqref{b0}, \eqref{bn}, and \eqref{def}. Therefore, if $q$ satisfies Assumption~\ref{L1}, the function $\bar{P}$ takes values in $[0,1]$ and satisfies properties (i)-(iii) from the definition of a transition function.

\begin{theorem}\rm{(cp.  Feinberg et al.~\cite[Theorem~2.2]{FMS})}
\label{thm:JMP}
Given  a probability measure $\gamma$ on ${\bf X}$ and a $Q$-function $q$ satisfying Assumption~\ref{L1}, the jump process $\{\BB{X}_t: t \in [T_0, T_1[\} $ defined in \eqref{jump} is a jump Markov process with the transition function $\bar{P}$.
\end{theorem}
\begin{proof}
The statement of the theorem follows from the same arguments as in the proof of Feinberg et al.~\cite[Theorem~2.2]{FMS}, where the case $T_0=0$ and $T_1=\infty$ was considered. 
We remark that though it was assumed there that the $Q$-function  $q$  satisfies Assumption~\ref{LB},  the arguments in the proof in \cite{FMS} only require that,
\begin{equation}
\label{eq-a}
\int_u^t q(x,s)ds < \infty, \qquad x \in {\bf X},\  u\in [T_0,T_1[,\  t \in ]u,T_1[,
\end{equation}
and this holds in view of Assumption~\ref{L1}. \qed
\end{proof}

 Let $\cal P$ be the family of all real-valued non-negative functions $P(u,x;t,B),$ defined for all $t \in ]T_0, T_1[,$  $u \in [T_0, t[,$ $x \in {\bf X},$ and $B \in \B({\bf X}),$  which are measurable  in $(u,x)\in [T_0, t[\times {\bf X}$ for all $t\in ]T_0, T_1[$ and $B\in \B({\bf X}).$ Observe  that $\bar{P} \in {\cal P}.$  

Consider a set $E$ and some family $\cal A$ of functions $f:E\to\bar{\BB{R}}=[-\infty,+\infty].$  A function $f$ from $\cal A$ is called minimal in the family $\cal A$ if for every function $g$ from $\cal A$ the inequality $f(x)\le g(x)$ holds  for all $x\in E.$ The following theorem generalizes   Theorems 3.1 and 3.2 in Feinberg et al. ~\cite{FMS}, stating the same statements under Assumption~\ref{LB} which is stronger than Assumption~\ref{L1}.    The measurability in $(u,x)$ of a function satisfying Kolmogorov's backward equation is implicitly assumed in Feinberg et al.~\cite{FMS}.
\begin{theorem}
\label{thm:BKE}
Under Assumption~\ref{L1}, the transition function $\bar{P}$ is the minimal function in $\cal P$ satisfying  the following two properties:

(i) for all $t\in]T_0,T_1[,$  $x\in {\bf X},$ and $B\in\B({\bf X}),$
\begin{equation}
\label{BC1}
\lim\limits_{u \to t-}P(u,x;t,B) = {\bf I} \{ x \in B \},
\end{equation}
 and the function is absolutely continuous in $u \in [T_0, t[;$

(ii) for all $t\in]T_0,T_1[,$  $x\in {\bf X},$ and $B\in\B({\bf X}),$ Kolmogorov's backward equation
\begin{equation}
\label{eq:BKDE}
\frac{\partial}{\partial u}{P}(u,x;t,B) =
q(x,u){P}(u,x;t,B)
- \int_{{\bf X} }q(x,u,dy \setminus
\{x\})P(u,y;t,B) 
\end{equation} holds for almost every $u\in [T_0,t[.$

 In addition, if the transition function  $\bar{P}$ is  regular (that is, $\bar{P}(u,x;t,{\bf X}) $ $=1$ for all
$u,$ $x,$ $t$ in the domain of $\bar P$), then $\bar{P}$ is the unique
function in $\cal P$  satisfying  properties (i), (ii) 
and which is
 a measure on $({\bf X}, \B({\bf X}))$ for all $t\in]T_0,T_1[,$ $u\in [T_0,t[,$ and  $x\in {\bf X},$
 and taking values in $[0,1]$.
\end{theorem}
\begin{proof}
Under Assumption~\ref{LB}, this theorem is Theorems~3.1 and 3.2 from Feinberg et al.~\cite{FMS} combined. However, the proofs there only use the property  that \eqref{eq-a} holds,  and this property is true under Assumption~\ref{L1}. Therefore, the statement of the theorem holds. \qed
\end{proof}

\section{Kolmogorov's forward equation}
\label{S-FE}

Under Assumption~\ref{LB}, Kolmogorov's forward equation~\eqref{eq:FKDE} was studied by Feller~\cite[Theorem~1]{Fel} for continuous $Q$-functions and by Feinberg et al.~\cite[Theorems~4.1, 4.3]{FMS} for measurable $Q$-functions. In this section, we study Kolmogorov's forward equation~\eqref{eq:FKDE} under Assumption~\ref{ALB},
which, in view of Lemma~\ref{l:A-eq}(b), is more general than Assumption~\ref{LB}.

Let $\hat{\cal P}$ be the family of real-valued functions $\hat{P}(u,x;t,B),$  defined for all $u \in [T_0, T_1[$, $t \in ]u, T_1[$, $x \in {\bf X},$ and $B \in \B({\bf X}),$  which are measures on $({\bf X}, \B({\bf X}))$ for fixed $u,$ $x,$ $t$ and are measurable functions in $t$ for fixed $u,$ $x,$ $B.$  In particular, $\bar P\in \hat{\cal P},$ where $\bar P$ is defined in \eqref{def}.

\begin{definition}
\label{defXnt}
For $s \in ]T_0, T_1],$ a set $B\in \mathfrak{B}({\bf X})$ is  called $(q,s)$-bounded if the function $q(x,t)$ is bounded on the set $B\times [T_0,s[.$ 
\end{definition}

\begin{definition}
\label{defbbf}
A $(q,T_1)$-bounded set  is called $q$-bounded.
\end{definition}
  In Definition~\ref{defbbf} we follow the terminology from Feinberg et al.~\cite[p.~262]{FMS}. Feller~\cite{Fel} called such sets {\it bounded}. 

The following theorem  shows that the transition function $\bar{P}$ is the minimal function satisfying Kolmogorov's forward equation.  
Being applied to a function $q$ satisfying  the  stronger Assumption~\ref{LB}, this theorem implies Corollary~\ref{CORR2}, which is a stronger  result than \cite[Theorem 4.3]{FMS}; see explanations before Corollary~\ref{CORR2}.

\begin{theorem}
\label{thm:FKDE}
Under Assumption~\ref{ALB}, the transition function $\bar{P}$ is the minimal function in $\hat{\cal P}$ satisfying  the following two properties:

(i) for all $u \in [T_0, T_1[,$   $s \in ]u,T_1[,$  $x \in {\bf X},$ and  $(q,s)$-bounded  sets $B,$
\begin{equation}
\label{BC2}
\lim_{t \to u+} P(u,x;t,B) = {\bf I}\{x \in B\},
\end{equation}
and the function is absolutely continuous in $t \in ]u,s[;$

(ii) for all $u \in [T_0, T_1[,$   $s \in ]u,T_1[,$  $x \in {\bf X},$ and  $(q,s)$-bounded  sets $B,$ Kolmogorov's forward equation
\begin{equation}
\label{eq:FKDE}
\frac{\partial}{\partial t} P(u,x;t,B) = -\int_{B} q(y,t) P(u,x;t,dy)
 + \int_{\bf X} q(y,t,B\setminus \{y\})P(u,x;t,dy),
\end{equation}
holds for almost every $t \in ]u,s[.$

 In addition, if the transition function  $\bar{P}$ is  regular, then $\bar{P}$ is the unique
function in $\hat{\cal P}$  satisfying properties (i), (ii) 
and taking values in $[0,1]$.

\end{theorem}

As  stated in Theorem~\ref{thm:FKDE}, the function $\bar{P}$ satisfies Kolmogorov's forward equation~\eqref{eq:FKDE} for $(q,s)$-bounded sets $B \in \B({\bf X})$. In general, as the following example demonstrates, it is not possible to extend \eqref{eq:FKDE} to all sets $B \in \B({\bf X})$.
\begin{example}\label{ex}
\emph{For a set $B\in\B({\bf X}),$ Kolmogorov's forward equation~\eqref{eq:FKDE} does not hold at all  $t \in ]u,T_1[$.} {\rm Let ${\bf X} = \BB{Z},$ where $\BB{Z}$ denotes the set of integers, $q(0, t) = 1,\  q(0, t, j)= 2^{-(|j| + 1)}$ for all $j \ne 0,$  and $q(j, t, -j) = q(j, t) =  2^{|j|}$ for all $ j \ne 0.$  If ${\BB X}_u=0,$ then starting at time $u$ the process spends an exponentially distributed amount of time at state 0, then it jumps to a state $j\ne 0$ with probability $2^{-(|j|+1)},$ and then it oscillates between the states $j$ and $-j$ with equal intensities.   Thus for all $u \in [T_0, T_1[$ and $t \in ]u,T_1[$
\[\bar{P}(u,0;t,0) = e^{-(t-u)} \qquad \text{ and } \qquad \bar{P}(u,0;t,j)  = \frac{1- e^{-(t-u)}}{2^{|j|+1}}, \qquad j \ne 0,\]
which implies 
\begin{multline*}
\int_{\bf X} q(y,t, {\bf X} \setminus \{y\})\bar{P}(u,0;t,dy)= \int_{\bf X} q(y,t)\bar{P}(u,0;t,dy) \\
= q(0,t)\bar{P}(u,0;t,0) + \sum_{j \ne 0} q(j,t)\bar{P}(u,0;t,j)= e^{-(t-u)} + \sum_{j > 0} (1- e^{-(t-u)}) = \infty.
\end{multline*}
Thus, if $B ={\bf X}$, then \eqref{eq:FKDE} does not hold with $P=\bar P$ because both integrals in \eqref{eq:FKDE} are infinite.\qed}
\end{example}

The following theorem describes the necessary and sufficient condition for a function $P$ from $\hat{\cal P}$ to satisfy properties (i) and (ii)  stated in Theorem~\ref{thm:FKDE}. In other words, it provides a necessary and sufficient condition that a function $P$ from $\hat{\cal P}$  satisfies Kolmogorov's forward equation. The  necessity part of this theorem plays the central role in proving the minimality property of $\bar P$ stated in Theorem~\ref{thm:FKDE}.
\begin{theorem}
\label{thm:intFKE}
 Let Assumption~\ref{ALB} hold. 
  A function $P$ from $\hat{\cal P}$ satisfies properties (i) and (ii)  stated in Theorem~\ref{thm:FKDE} if and only if, for all $u \in [T_0, T_1[,$   $t \in ]u,T_1[,$  $x \in {\bf X},$ and  $B \in \B({\bf X}),$ 
\begin{equation}
\label{int-FKE}
\begin{split}
P(u,x;t,B) &= {\bf I}\{x \in B\}e^{-\int_u^t q(x,\theta)d\theta} \\
&\qquad \qquad  + \int_u^t \int_{\bf X} \int_B e^{-\int_w^t q(y,\theta)d\theta} q(z,w, dy\setminus \{z\}) P(u,x;w,dz)  dw.
\end{split}
\end{equation}

\end{theorem}

\begin{lemma}
\label{lem:FKDE-sol}
Under Assumption~\ref{ALB}, the following statements hold:

(a) for each $u \in [T_0, T_1[$, $s\in ]u, T_1[,$  $x \in {\bf X}$, and $B \in \B({\bf X})$, the function $\bar{P}(u,x;t,B) $  
 satisfies the boundary condition \eqref{BC2} and is absolutely continuous in $t \in ]u,s[.$

 (b) 
  the function $\bar{P}$ satisfies  property  (ii)  stated in Theorem~\ref{thm:FKDE}. 
\end{lemma}
\begin{proof}
 (a) Under Assumption~\ref{LB}, statement (a) of this lemma is Theorem~4.1(i) in Feinberg et al.~\cite{FMS}, and the proof there is correct if \eqref{eq-a} holds. In view of Lemma~\ref{l:A-eq}(c),  formula~\eqref{eq-a} is true under Assumption~\ref{ALB}, and therefore, statement (a) of the lemma holds.

(b) Fix an arbitrary $s \in ]T_0, T_1[.$ Observe that a $Q$ function satisfying Assumption~\ref{ALB} satisfies Assumption~\ref{LB} with $T_1=s$. Then  it follows from Feinberg et al.~\cite[Theorem~4.1(ii)]{FMS} that, for all $u \in [T_0, s[$, $x \in {\bf X}$, and $(q,s)$-bounded sets $B \in \B({\bf X})$, the function $\bar{P}(u,x;t,B)$ satisfies Kolmogorov's forward equation~\eqref{eq:FKDE} for almost every $t \in ]u,s[.$ Since $s$ was chosen arbitrarily, this fact implies that the function $\bar{P}$ satisfies property (ii)   stated in Theorem~\ref{thm:FKDE}. 
\qed
\end{proof}

To prove Theorems~\ref{thm:FKDE} and \ref{thm:intFKE} we formulate and prove two lemmas. Lemmas~\ref{Cor} and ~\ref{l:int-S} present Kolmogorov's forward equation in integral  forms~\eqref{int-FKE},~\eqref{eq:FKE}, which are equivalent to its differential form~\eqref{eq:FKDE}. In particular, Theorem~\ref{thm:intFKE} follows from Lemma~\ref{l:int-S}.
  Let $u \in [T_0, T_1[,$  $s \in ]u,T_1[$, $ x \in {\bf X},$ and $B \in \B({\bf X})$ be a $(q,s)$-bounded set.
For any function $P$ from $\hat{\cal P},$ 
\begin{equation}
\label{almostM}
\int_{B} q(y,t)P(u,x;t,dy) \le \left (\sup_{y \in B, t \in ]u,s[}q(y,t)\right) P(u,x;t,B) < \infty, \qquad   t \in ]u,s[.
\end{equation}
In addition, for $u,$ $s,$ $x,$ and $B$ described above, if the function $P$ satisfies the boundary  condition \eqref{BC2} and is absolutely continuous in $t \in ]u,s[,$ then it is bounded in $t \in ]u,s[,$ which along with \eqref{almostM} implies that
\begin{equation}
\label{eq:HEF2M}
\int_u^t \int_{B} q(y,w)P(u,x;w,dy)dw < \infty, \qquad \qquad \qquad t \in ]u,s[.
\end{equation}
\begin{lemma}
\label{Cor}
For arbitrary fixed $u \in [T_0, T_1[$, $s \in ]u, T_1[$, $x \in {\bf X}$, and $(q,s)$-bounded set $B \in \B({\bf X}),$ 
 a function $P$ from $\hat{\cal P}$ satisfies the  equality
\begin{equation}
\label{eq:FKE}
\begin{split}
P&(u,x;t,B)= {\bf I}\{x \in B\} \\
&- \int_{u}^{t} \int_{B} q(y,w)P(u,x;w,dy)dw  + \int_u^t \int_{\bf X} q(y,w,B\setminus \{y\}) P(u,x; w, dy)  dw, \quad t \in ]u,s[,
\end{split}
\end{equation}
 if and only if it satisfies
the boundary  condition \eqref{BC2}, is absolutely continuous in $t \in ]u,s[,$ and satisfies Kolmogorov's forward equation~\eqref{eq:FKDE}  for almost every $t\in ]u,s[.$
\end{lemma}
\begin{proof}
Suppose that  a function $P$ from $\hat{\cal P}$  satisfies  the boundary condition \eqref{BC2}, is absolutely continuous in $t \in ]u,s[,$ and satisfies Kolmogorov's forward equation~\eqref{eq:FKDE}  for almost every $t\in ]u,s[$  for  $u,$ $s,$ $x,$ and $B$ described in the formulation of the lemma.
Since every absolutely continuous function is the 
integral of its derivative, equality 
\eqref{eq:FKE} follows from  integrating equation \eqref{eq:FKDE} from $u$ to $t$ and using the boundary condition \eqref{BC2}. 
In particular, both integrals in equality \eqref{eq:FKE} are finite because, in view of \eqref{eq:HEF2M}, the first integral is finite.  

Now,  suppose \eqref{eq:FKE} holds for  $u,$ $s,$ $x,$ and $B$ described in the formulation of the lemma. Observe that, for fixed  $u,$ $s,$ $x,$ and $B$, the real valued function $P(u,x;t,B)$ is  a constant plus the the difference of two  integrals from $u$ to $t$  of nonnegative integrable functions defined for $w \in ]u,s[$.  Since an  integral of an integrable function is an absolutely continuous function of the upper limit of integration   
and its derivative is equal to the integrand almost everywhere on its domain (Royden~\cite[Thms~10 on p. 107  and 14 on p. 110]{Roy}), 
the function $P(u,x;t,B)$ is 
 absolutely continuous in $t \in ]u,s[,$ and Kolmogorov's forward equation~\eqref{eq:FKDE} holds for almost every $t\in ]u,s[$  for the fixed $u,$ $s,$ $x,$ and $B$. In addition, the absolute continuity of the integrals in  \eqref{eq:FKE} implies that  \eqref{BC2} holds. 
  \qed
\end{proof}

\begin{lemma}
\label{l:int-S}
Let $u \in [T_0, T_1[$, $s \in ]u, T_1[$, $x \in {\bf X}$, and $C \in \B({\bf X})$ be a $(q,s)$-bounded set. A function $P$ from $\hat{\cal P}$ satisfies  for all $B \in \B(C)$ the boundary condition  \eqref{BC2}, is absolutely continuous in $t \in ]u,s[,$ and satisfies Kolmogorov's forward equation~\eqref{eq:FKDE}  for almost every $t\in ]u,s[$
  if and only if  it satisfies equality \eqref{int-FKE}  
  for all $t \in ]u,s[$ and $B \in \B(C).$
\end{lemma}

\begin{remark}
\label{T-eq2}
The sufficiency statements of Theorem~\ref{thm:intFKE} and Lemma~\ref{l:int-S} are not used in the proofs in this section.
\end{remark}
\begin{proof}\emph{of Lemma~\ref{l:int-S}} 
%
%
 The following version of Fubini's  theorem from Halmos~\cite[Section~36, Remark~(3)]{Halmos} is used in the proof. Let  $(Z, {\bf S}, \mu)$ be   a measure space   with $\mu(Z) < \infty,$ and let $(Y, {\bf T})$  be a measurable space.   Suppose that to almost every $z \in Z$ there corresponds  a finite measure $\nu_z$ on ${\bf T}$ such that the function $\phi(z):=\nu_z(B)$ is measurable in $z$ for each measurable subset $B$ of $Y.$ Then, for any non-negative measurable function $g$ on $Y$,
\begin{equation}
\label{eq:exchange}
\int_{Z} \left(\int_Y g(y) \nu_z(dy) \right) \mu(dz) = \int_Y g(y) \nu(dy),
\end{equation}
where, for each measurable subset $B$ of $Y$,
\[\nu(B):= \int_Z \nu_z(B) \mu(dz).\]

 Let us fix $ P\in\hat{\cal P},$  $u \in [T_0, T_1[$, $s \in ]u, T_1[$, $x \in {\bf X}$, and a $(q,s)$-bounded set $C \in \B({\bf X})$. To simplify notations,  define
\begin{align}
\label{G-1}
G^{(1)}(t,B) &:= \int_{\bf X} q(z,t,B \setminus \{z\}) {P}(u,x;t,dz), \qquad & t \in ]u,s[,\ B \in \B(C),\\
\label{G-2}
G^{(2)}(t,B) &:= \int_{\bf X} {\delta}_z(B)q(z,t) {P}(u,x;t,dz),\qquad & t \in ]u,s[,\ B \in \B(C),
\end{align}
where $\delta_z(\cdot)$ is the Dirac measure on $({\bf X}, \B({\bf X})),$ 
\begin{equation}
\label{DM}
\delta_z(B) := {\bf I}\{z \in B\}, \qquad B \in \B({\bf X}).
\end{equation}
Observe that, for $j = 1,2,$  the function $G^{(j)}(t,\cdot)$ is a measure on $(C, \B(C))$ for every $t \in ]u,s[,$ and   $G^{(j)}(\cdot, B)$ is a measurable function on $]u,s[$  for every $B \in \B(C).$

Let  $t \in ]u,s[,$ $v \in ]u,t[,$ and $B \in \B(C).$ Consider $(Z,{\bf S}, \mu)=({\bf X}, \B({\bf X}), {P}(u,x;v,\cdot))$ and
$(Y, {\bf T}):=(C,\B(C)).$ For $\nu_z(\cdot) = q^+(z, v, \cdot),$ which is finite for all $z \in Z$ since $q$ is a $Q$-function, and for $g(y) = {\bf I}\{ y \in B\}e^{-\int_v^t q(y,\theta)d\theta},$  formula \eqref{eq:exchange} yields
\begin{equation}
\label{ex1}
\int_{\bf X} \left(\int_B e^{-\int_v^t q(y,\theta)d\theta} q(z,v, dy\setminus \{z\})\right) {P}(u,x;v,dz) = \int_B e^{-\int_v^t q(y,\theta)d\theta} G^{(1)}(v,dy).
\end{equation}

 \emph{Necessity.} For all $B \in \B(C)$, let the function $P$ satisfy
  the boundary condition  \eqref{BC2}, be absolutely continuous in $t \in ]u,s[,$ and satisfy Kolmogorov's forward equation~\eqref{eq:FKDE}  for almost every $t\in ]u,s[.$
Equation~\eqref{eq:FKDE} can be rewritten as
\begin{equation}\label{eq:FKEShort}
\frac{\partial}{\partial t} P(u,x;t,B) = -G^{(2)}(t,B)+G^{(1)}(t,B).
\end{equation}
Formula \eqref{almostM} means that $G^{(2)}(t,B)<\infty$ for all $t\in ]u,s[.$ This inequality and \eqref{eq:FKEShort} imply that, for $j = 1,2,$
 \begin{equation}
 \label{almost1}
   G^{(j)}(t, C)< \infty \quad \text{ for almost every } \quad t \in ]u,s[.
 \end{equation}

For $j = 1,2,$ consider the non-negative functions $H^{(j)}:(]u,s[ \times \B(C))\to \BB{R}_+,$
\begin{equation}
\label{eq:H}
H^{(j)}(t,B) := \int_u^t G^{(j)}(w,B) dw, \qquad t \in ]u,s[,\ B \in \B(C).
\end{equation}
In view of Lemma~\ref{Cor},
\begin{equation}
\label{eq:FKE-alt}
{P}(u,x;t,B) = {\bf I}\{x \in B\} + H^{(1)}(t,B)- H^{(2)}(t,B), \qquad t \in ]u,s[,\ B \in \B(C).
\end{equation}
Equality \eqref{eq:HEF2M}, which implies \eqref{eq:HEF} for $j=2,$  and \eqref {eq:FKE-alt} yield
\begin{equation}
\label{eq:HEF}
H^{(j)}(t,B) <\infty, \qquad\qquad j=1,2,\ t \in ]u,s[,\ B \in \B(C).
\end{equation}

Observe that, for any measure $p(\cdot)$ on $(C, \B(C))$ and $w \in [u,t[$,
\begin{equation}
\label{0G}
\begin{aligned}
 \int_B (1-e^{-\int_w^t q(y, \theta)d\theta})p(dy) &= \int_B \left (\int_w^t q(y,v) e^{-\int_v^t q(y,\theta)d\theta} dv\right)p(dy) \\
&= \int_w^t \int_B q(y,v) e^{-\int_v^t q(y,\theta)d\theta}  p(dy) dv,
\end{aligned}
\end{equation}
where the first equality is correct since
\begin{equation}
\label{exp}
\int_w^t q(y,v)e^{-\int_v^t q(y,\theta)d\theta} dv  = 1 - e^{-\int_w^t q(y,\theta)d\theta}, \qquad y \in {\bf X},
\end{equation}
and the last one is obtained by changing the order of integration in $y$ and $v$ and applying Fubini's theorem. Let $j = 1,2,$ $t \in ]u,s[$, and $B \in \B(C).$  Then
\begin{multline}
\label{GH1}
H^{(j)}(t, B) - \int_u^t \int_B e^{-\int_w^t q(y,\theta)d\theta} G^{(j)}(w,dy) dw  =  \int_u^t \int_B (1- e^{-\int_w^t q(y,\theta)d\theta}) G^{(j)}(w,dy) dw\\
\begin{aligned}
&=\int_u^t \left( \int_w^t  \int_B q(y,v) e^{-\int_v^t q(y,\theta)d\theta} G^{(j)}(w,dy) dv\right)dw  \\
&= \int_u^t \int_u^v \left (\int_B q(y,v) e^{-\int_v^t q(y,\theta)d\theta} G^{(j)}(w,dy) \right) dw dv,\\
&= \int_u^t \int_B q(y,v) e^{-\int_v^t q(y,\theta)d\theta} H^{(j)}(v,dy)dv,
\end{aligned}
\end{multline}
where the first equality follows from \eqref{eq:H}, the second equality follows from \eqref{0G} with  $p(\cdot) = G^{(j)}(w,\cdot),$  the third equality is obtained by changing the order of integration in $w$ and $v$, and the last one is obtained from formula \eqref{eq:exchange} by setting $(Z, {\bf S}, \mu):=(]u,v[,\B(]u,v[),\lambda),$ where $\lambda$ is the Lebesgue measure,
$(Y, {\bf T}):=(C,\B(C)),$ $\nu_z(\cdot) = G^{(j)}(z, \cdot),$ which, in view of inequality \eqref{almost1} is finite for almost every $z \in Z$, and $g(y) = {\bf I} \{y \in B\}q(y,v) e^{-\int_v^t q(y,\theta)d\theta}.$

For $v \in ]u,t[$, by setting $(Z,{\bf S}, \mu):=({\bf X}, \B({\bf X}), P(u,x;v,\cdot))$, $(Y, {\bf T}):=(C,\B(C))$, $\nu_z(\cdot): = q(z, v)\delta_z(\cdot)$, and  $g(y): = {\bf I}\{ y \in B\}e^{-\int_v^t q(y,\theta)d\theta},$  formula \eqref{eq:exchange}  yields
\begin{equation}
\label{ex2}
\int_B e^{-\int_v^t q(y,\theta)d\theta} G^{(2)}(v,dy) = \int_{\bf X} \left(\int_B e^{-\int_v^t q(y,\theta)d\theta} q(z,v) \delta_z(dy)\right) {P}(u,x;v,dz).
\end{equation}
 Therefore, for all $t \in ]u,s[$ and $B \in \B(C),$
\begin{align}
\label{2}
&H^{(2)}(t,B) = \int_u^t \int_B e^{-\int_v^t q(y,\theta)d\theta} G^{(2)}(v,dy) dv + \int_u^t \int_B q(y,v) e^{-\int_v^t q(y,\theta)d\theta} H^{(2)}(v,dy) dv \notag\\
&= \int_u^t \int_B q(y,v)e^{-\int_v^t q(y,\theta)d\theta} {P}(u,x;v,dy) dv + \int_u^t \int_B q(y,v) e^{-\int_v^t q(y,\theta)d\theta} H^{(2)}(v,dy) dv \notag\\
&= \int_u^t \int_B q(y,v)e^{-\int_v^t q(y,\theta)d\theta} \delta_x(dy)dv + \int_u^t \int_B q(y,v)e^{-\int_v^t q(y,\theta)d\theta} H^{(1)}(v, dy)dv,\\
&= \left({\bf I}\{x \in B\} - {\bf I}\{x \in B\}e^{-\int_u^t q(x, \theta)d\theta}\right) + \left (H^{(1)}(t,B) - \int_u^t \int_B e^{-\int_v^t q(y,\theta)d\theta} G^{(1)}(v,dy) dv \right)\notag\\
&= {\bf I}\{x \in B\} + H^{(1)}(t, B) \notag\\
&\qquad - \left( {\bf I}\{x \in B\} e^{-\int_u^t q(x, \theta)d\theta} + \int_u^t \int_{\bf X} \int_B e^{-\int_v^t q(y,\theta)d\theta} q(z,v, dy\setminus \{z\}) {P}(u,x;v,dz) dv \right)\notag,
\end{align}
where the first equality follows from \eqref{GH1} with $j = 2$, the second equality follows from \eqref{DM} and \eqref{ex2}, the third equality follows from \eqref{DM}, \eqref{eq:FKE-alt}, and \eqref{eq:HEF}, the fourth equality follows from \eqref{0G} with $p(\cdot) = \delta_x(\cdot)$ and $w = u$ and \eqref{GH1} with $j = 1$, and the last one follows from \eqref{ex1}. Thus, \eqref{eq:FKE-alt} and \eqref{2} imply \eqref{int-FKE}.

\emph{Sufficiency.} Assume that the function ${P}$ satisfies \eqref{int-FKE} for all $t \in ]u,s[$ and $B \in \B(C)$.   As follows from Lemma~\ref{Cor}, it is sufficient to show that \eqref{eq:FKE} holds for all $B \in \B(C).$
In view of equality \eqref{ex1}, formula~\eqref{int-FKE} can be rewritten as
\begin{equation}
\label{int-FKEs}
{P}(u,x;t,B) = {\bf I}\{x \in B\}  e^{-\int_u^t q(x, \theta)d\theta} + \int_u^t \int_B e^{-\int_v^t q(y,\theta)d\theta} G^{(1)}(v,dy) dv.
\end{equation}
Let $t \in ]u,s[$ and $B \in \B(C).$ Since
\begin{equation}
\label{0G1}
\int_w^t q(y,v)e^{-\int_w^v q(y,\theta)d\theta} dv = 1- e^{-\int_w^t q(y,\theta)d\theta}, \qquad y \in {\bf X},\ w \in [u,t[,
\end{equation}
it follows from \eqref{0G1} and Fubini's theorem that, for any measure $p(\cdot)$ on $(C, \B(C)),$ 
\begin{equation}
\label{exp1}
\int_B (1- e^{-\int_w^t q(y,\theta)d\theta})p(dy) = \int_w^t \int_B q(y,v) e^{-\int_w^v q(y,\theta)d\theta}p(dy) dv, \quad w \in [u,t[.
\end{equation}
Observe that formula~\eqref{exp1} differs from \eqref{0G}.  Next,
\begin{multline}
\label{22}
 \int_u^t  G^{(1)}(w,B) dw - \int_u^t \int_B e^{-\int_w^t q(y,\theta)d\theta} G^{(1)}(w,dy) dw  \\
\begin{aligned}
&= \int_u^t \left (\int_w^t \int_B q(y,v)  e^{-\int_w^v q(y,\theta) d\theta} G^{(1)}(w,dy)dv \right)  dw\\
&= \int_u^t \int_u^v \left (\int_B q(y,v)  e^{-\int_w^v q(y,\theta) d\theta} G^{(1)}(w,dy) \right)dw dv\\
&= \int_u^t \left(\int_B  q(y,v) \int_u^v   e^{-\int_w^v q(y,\theta) d\theta} G^{(1)}(w,dy) dw \right) dv,
\end{aligned}
\end{multline}
where the first equality follows from  \eqref{exp1} with $p(\cdot) = G^{(1)}(w,\cdot)$, the second equality is obtained by interchanging the order of integration in $w$ and $v$, and the last one is obtained from \eqref{eq:exchange} by setting $(Z,{\bf S}, \mu)=(]u,v[,\B(]u,v[),\lambda),$ where $\lambda$ is the Lebesgue measure, $(Y, {\bf T}):=(C,\B(C))$, $\nu_z(B) = \int_B e^{-\int_z^v q(y, \theta)d\theta}G^{(1)}(z,dy)$, which, in view of   equality~\eqref{int-FKEs} and the property that the function $P$ takes values in $[0, \infty[,$ is finite for $B=C$ and for almost every $z\in Z$,  and  $g(y) = q(y,v){\bf I}\{y \in B\}.$ Therefore,
\begin{multline*}
\begin{aligned}
{P}(u,x;t,B) &= {\bf I}\{x \in B\}  - \int_u^t \int_B q(y,v) e^{-\int_u^v q(y,\theta)d\theta} \delta_x(dy) dv \\
&+ \int_u^t G^{(1)}(w,B)dw - \int_u^t \left(\int_B  q(y,v) \int_u^v   e^{-\int_w^v q(y,\theta) d\theta} G^{(1)}(w,dy) dw \right) dv \\
&= {\bf I}\{x \in B\} + \int_u^t G^{(1)}(w,B)  - \int_u^t \int_B q(y,v){P}(u,x;v,dy)dv,
\end{aligned}
\end{multline*}
where the first equality follows from \eqref{int-FKEs}, \eqref{exp1} with $p(\cdot) = \delta_x(\cdot)$,  and \eqref{22}, and the last one is obtained by substituting ${P}(u, x; v,dy)$ with \eqref{int-FKEs}. Thus, it follows from \eqref{G-1} and the above equality that \eqref{eq:FKE} holds for all $B \in \B(C).$
\qed
\end{proof}

\begin{proof} \emph{of Theorem~\ref{thm:intFKE}}
The sufficiency statement of the theorem follows immediately from Lemma~\ref{l:int-S}, and the necessity statement of the theorem follows from Lemma~\ref{l:int-S} and Lebesgue's monotone convergence theorem, as explained below.

{\it Necessity. } Assume that, for all $u \in [T_0, T_1[,$ $s \in ]u, T_1[,$ $x \in {\bf X}$, and $(q,s)$-bounded sets $C$,  properties  (i) and (ii)  stated in Theorem~\ref{thm:FKDE} hold for the function $P$. Assumption~\ref{ALB} and Lemma~\ref{l:A-eq}(a) imply that for each $s\in ]T_0,T_1[ $ there exist   $(q,s)$-bounded sets $B^s_1,B^s_2,\ldots$  such that $B^s_n\uparrow {\bf X}$ as $n \to \infty$.
   Then, for all  $u \in [T_0, T_1[,$ $s \in ]u, T_1[,$ $t \in ]u,s[,$ $x \in {\bf X}$, and $B \in \B({\bf X}),$
\begin{multline}
\begin{aligned}
P(&u,x;t,B)=\lim_{n\to\infty} P(u,x;t,B\cap B^s_n) = \lim_{n\to\infty} {\bf I}\{x \in B\cap B^s_n\}e^{-\int_u^t q(x,\theta)d\theta} \\
&\qquad \qquad+ \lim_{n\to\infty}  \int_u^t \int_{\bf X} \int_{B\cap B^s_n} e^{-\int_w^t q(y,\theta)d\theta} q(z,w, dy\setminus \{z\}) P(u,x;w,dz)  dw\\
&= {\bf I}\{x \in B\}e^{-\int_u^t q(x,\theta)d\theta} + \int_u^t \int_{\bf X} \int_{B} e^{-\int_w^t q(y,\theta)d\theta} q(z,w, dy\setminus \{z\}) P(u,x;w,dz)  dw,
\end{aligned}
\end{multline}
where the first equality is correct since the sets $B_n^s \uparrow {\bf X}$ as $n \to \infty$, the second equality follows from Lemma~\ref{l:int-S}, and the last one follows from Lebesgue's monotone convergence theorem since the sets $B_n^s \uparrow {\bf X}$ as $n \to \infty$. Since the above equality holds for all $t \in ]u,s[$ for each $s \in ]u,T_1[$, formula~\eqref{int-FKE} holds for all  $u \in [T_0, T_1[,$ $t \in ]u,T_1[,$ $x \in {\bf X}$, and $B \in \B({\bf X}).$ \qed
\end{proof}

\begin{proof} \emph{of Theorem~\ref{thm:FKDE}}  In view of Lemma~\ref{lem:FKDE-sol}, we need  to prove only the minimality and uniqueness properties of
$\bar P$  among functions from $\hat{\cal P}$  satisfying  properties (i) and (ii) stated in Theorem~\ref{thm:FKDE}.
Let $ P$ be a function from $\hat{\cal P}$ satisfying these properties.
Let $u \in [T_0, T_1[,$ $t \in ]u, T_1[,$ $x \in {\bf X}$, and $B \in \B({\bf X}).$ In view of Theorem~\ref{thm:intFKE}, formula~\eqref{int-FKE} holds. Since the last term in \eqref{int-FKE} is non-negative,
\[{P}(u,x;t,B) \ge {\bf I}\{x \in B\}e^{-\int_u^t q(x, \theta)d\theta} = \bar{P}^{(0)}(u,x;t,B),\]
where the last equality is \eqref{b0}. Assume that for some $n=0,1,\ldots,$
\begin{equation}\label{eq:indn}
{P}(u,x;t,B) \ge \sum_{m=0}^n\bar{P}^{(m)}(u,x;t,B).
\end{equation}
Then, from \eqref{bn-alt}, \eqref{int-FKE}, and \eqref{eq:indn}, $ {P}(u,x;t,B) \ge \sum_{m=0}^{n+1}\bar{P}^{(m)}(u,x;t,B).$ Thus, by induction, \eqref{eq:indn} holds for all $n=0,1,\ldots\ .$   Let $n\to\infty.$ Then \eqref{eq:indn} and \eqref{def} imply that  ${P}(u,x;t,B) \ge \bar{P}(u,x;t,B).$ Therefore, the function $\bar{P}$ is the minimal function from $\hat{\cal P}$ satisfying  properties (i) and (ii) stated in Theorem~\ref{thm:FKDE}.

In conclusion, let the transition function $\bar P$ be  regular.   If there is another function  ${ P,}$ which satisfies  properties (i) and (ii)  stated in Theorem~\ref{thm:FKDE} and takes values in $[0,1]$, then, since $\bar P$ is the minimal solution,  ${P}(u,x;t,B)> {\bar P}(u,x;t,B)$  for some  $u \in [T_0, T_1[,$ $x \in {\bf X},$ $t \in ]u, T_1[,$  and $B \in \B({\bf X}).$
In addition, ${P}(u,x;t,{\bf X}\setminus B)\ge {\bar P}(u,x;t,{\bf X}\setminus B).$  Therefore, ${P}(u,x;t,{\bf X})= {P}(u,x;t,B)+ {P}(u,x;t,{\bf X}\setminus B)> \bar P(u,x;t,B)+\bar P(u,x;t,{\bf X}\setminus B)=\bar P(u,x;t,{\bf X})=1,$ and the inequality ${P}(u,x;t,{\bf X})>1$ contradicts the property that ${P}$ takes values in $[0,1].$\qed
\end{proof}

Theorems~\ref{thm:FKDE} and \ref{thm:intFKE} imply the following two corollaries.

\begin{corollary}
\label{Cor-G}
Under Assumption~\ref{ALB}, the following statements hold:

(a) for all  $u \in [T_0, T_1[,$ $s \in ]u,T_1[,$ $x \in {\bf X},$ and $(q,s)$-bounded  sets $B,$ the function \\ $\bar{P}(u,x;t,B)$ satisfies \eqref{eq:FKE}. 

(b) the function $\bar{P}$ is the minimal function in $\hat{\cal P}$ for which statement~(a) holds. In addition, if the transition function $\bar{P}$ is regular, then $\bar{P}$ is the unique function in $\hat{\cal P}$ with values in $[0,1]$ for which statement~(a) holds.
\end{corollary}
\begin{proof}
In view of Lemma~\ref{Cor}, any  function $P$ from ${\hat{\cal P}}$ satisfies statement (a) of the corollary if and only if it satisfies properties (i) and (ii)  stated in  Theorem~\ref{thm:FKDE}.  Thus, the corollary follows from Theorem~\ref{thm:FKDE}. \qed
\end{proof}

\begin{corollary}
\label{cor:intFKE}
Let Assumption~\ref{ALB} hold. 
 The function $\bar P$ is the minimal function $P$ in  $\hat{\cal P}$ satisfying   equality~\eqref{int-FKE} for all $u \in [T_0, T_1[,$   $t \in ]u,T_1[,$  $x \in {\bf X},$ and  $B \in \B({\bf X}).$ In addition, if the transition function $\bar{P}$ is regular, then $\bar{P}$ is the unique function in $\hat{\cal P}$ with values in $[0,1]$ satisfying  equality~\eqref{int-FKE} for all $u \in [T_0, T_1[,$   $t \in ]u,T_1[,$  $x \in {\bf X},$ and  $B \in \B({\bf X}).$
 \end{corollary}
 \begin{proof}
  The  corollary follows from Theorems~\ref{thm:FKDE} and \ref{thm:intFKE}. \qed
 \end{proof}

\section{Kolmogorov's forward equation for $Q$-functions bounded at each state}
\label{S-GB}

This section  provides additional results  on Kolmogorov's forward equation when Assumption~\ref{LB} holds.  Under  Assumption~\ref{LB} Kolmogorow's forward  equation is studied in  Feinberg et al. \cite[Theorems~4.1, 4.3]{FMS}, and Corollary~\ref{CORR2} is a more general statement than  \cite[Theorem 4.3]{FMS}.   In addition,  Corollary~\ref{Cor-P}  describes
the minimality property  of the function $\bar P(T_0, x;t,B)$  that  is useful for applications to continuous-time Markov decision processes.  The following lemma and its corollary do not require any of Assumptions~\ref{Feller}-\ref{L1}.

\begin{lemma}
\label{LEMMABT}
Let $u\in [T_0,T_1[,$ $x\in{\bf X},$ and $B$ be a $q$-bounded set. A function $P \in \hat{\cal P}$ satisfies Kolmogorov's forward equation~\eqref{eq:FKDE} for almost every $t\in ]u,s[$ for all $s\in ]u,T_1[$ if and only if it satisfies this equation for almost every $t\in ]u,T_1[.$
\end{lemma}
%
%
\begin{proof}
The sufficiency statement of the lemma is straightforward since $]u,s[\subset ]u,T_1[$  when $s\in ]u,T_1[.$ Let a function $P \in \hat{\cal P}$ satisfy Kolmogorov's forward equation~\eqref{eq:FKDE} for almost every $t\in ]u,s[$ for all $s\in ]u,T_1[,$ where $u \in [T_0, T_1[,$ $x \in {\bf X},$ and $B \in \B({\bf X})$ is a $q$-bounded set. 
Consider an arbitrary  sequence $s_n\uparrow T_1$ as $n\to\infty$ with $s_1>u.$  Let $Y$ be the set of all $t\in ]u,T_1[$ such that \eqref{eq:FKDE} does not hold at point $t.$ 
For $n = 1,2,\ldots,$ the Lebesgue measure of the sets $Y\cap ]u,s_n[$ is 0 since each of these sets consists of points $t\in ]u,s_n[$ at which \eqref{eq:FKDE} does not hold. This implies that the Lebesgue measure of the set $Y$ is $0$. Therefore, the function $P$ satisfies Kolmogorov's forward equation for almost every $t \in ]u,T_1[.$ \qed
\end{proof}

\begin{corollary}\label{LEMMAB}
A function $P \in \hat{\cal P}$ satisfies for $q$-bounded sets $B$ properties (i) and (ii) stated in Theorem~\ref{thm:FKDE}  if and only if the following two properties hold:

(a) for all $u \in [T_0, T_1[$, $x \in {\bf X}$, and  $q$-bounded sets $B$,   the  function $P(u,x;t,B)$ 
satisfies the boundary condition \eqref{BC2} and is absolutely continuous in $t \in ]u,s[$ for each $s \in ]u,T_1[;$ 

(b) for all $u \in [T_0, T_1[$, $x \in {\bf X}$, and  $q$-bounded set $B$,   the function $P(u,x;t,B)$ satisfies Kolmogorov's forward equation~\eqref{eq:FKDE} for almost every $t\in ]u,T_1[.$
\end{corollary}
\begin{proof}
For $q$-bounded sets $B,$ property (i)  stated in Theorem~\ref{thm:FKDE} coincides with property (a) stated in the corollary. Lemma~\ref{LEMMABT} implies that property (ii) stated in Theorem~\ref{thm:FKDE} holds  for a  $q$-bounded set $B$ if and only if property (b) stated in the corollary holds. \qed
\end{proof}

\begin{lemma}\label{LEMMA5}
 Under  Assumption~\ref{LB}, a function  $P \in \hat{\cal P}$ satisfies properties (i) and (ii) stated in Theorem~\ref{thm:FKDE} if and only if it satisfies properties (a) and (b) stated in Corollary~\ref{LEMMAB}.

%
%
\end{lemma}
\begin{proof}
Let the function $P$ satisfy properties (i) and (ii) stated in Theorem~\ref{thm:FKDE}.  Since a $q$-bounded set is $(q,s)$-bounded, it follows from Corollary~\ref{LEMMAB} that  properties (a)  and (b) stated in Corollary~\ref{LEMMAB} hold.


Let properties (a) and (b) stated in Corollary~\ref{LEMMAB} hold.  Fix arbitrary $u \in [T_0, T_1[$, $s \in ]u, T_1[$, and $x \in {\bf X}.$  Lemma~\ref{l:int-S} implies that for every $q$-bounded set $B$   equality~\eqref{int-FKE} holds for all $t\in ]u,s[.$ In view of Assumption~\ref{LB} and  Lemma~\ref{l:A-eq}(a), there exist $q$-bounded  sets $B_1,B_2,\ldots$ such that $B_n\subseteq B_{n+1},$ $n=1,2,\ldots,$ and ${\bf X}=\cup_{n=1}^\infty B_n.$   Let $B \in \B(X)$.   Then $B^n:=B_n\cap B, $ $n=1,2,\ldots,$ are $q$-bounded sets.  Therefore,  for each set $B^n,$ equality~\eqref{int-FKE} holds for all $t \in ]u,s[.$ Since 
$B^n \uparrow B$ as $n \to \infty,$ Lebesgue's monotone convergence theorem implies that this formula also  holds for $B.$   Thus, in view of  Theorem~\ref{thm:intFKE}, the function $P$ satisfies properties  (i) and (ii) stated in Theorem~\ref{thm:FKDE}.\qed
\end{proof}

The following corollary generalizes Feinberg et al.~\cite[Theorem 4.1]{FMS} since Assumption~\ref{ALB} is weaker than Assumption~\ref{LB}. We remark that absolute continuity in $t\in ]u,\infty[$ in \cite[Theorem 4.1(i)]{FMS}  is meant in the sense that for each $s\in ]u,\infty[$ the function is absolutely continuous in $t\in ]u,s[.$ For $T_1=\infty$ this is equivalent  to the absolutely continuity assumed in property (a) stated in  Corollary~\ref{LEMMAB}. For unbounded intervals, this type of absolute continuity  is sometimes  called local absolute continuity.

\begin{corollary} {\rm (cp. Feinberg et al.~\cite[Theorem 4.1]{FMS})}
\label{CORR1}
Let Assumption~\ref{ALB} hold.   Then,  the function $\bar P$  satisfies properties (a) and (b) stated in Corollary~\ref{LEMMAB}.  In addition,  property (a) stated in Corollary~\ref{LEMMAB} holds for all $B\in\B({\bf X}).$
\end{corollary}
\begin{proof}
 In view of Lemma~\ref{lem:FKDE-sol}, the function $\bar{P}$ satisfies properties (i) and (ii) stated in Theorem~\ref{thm:FKDE}. In particular, it satisfies these properties for the smaller class of $q$-bounded sets. Thus, it follows from Corollary~\ref{LEMMAB} that the function $\bar{P}$ satisfies properties (a) and (b) stated in Corollary~\ref{LEMMAB}. In addition,  Lemma~\ref{lem:FKDE-sol}(a) implies that property (a) stated in Corollary~\ref{LEMMAB} holds for all $B\in\B({\bf X}).$  \qed
\end{proof}

The following corollary  generalizes 
\cite[Theorem 4.3]{FMS}.  The difference is that Corollary~\ref{CORR2} states  that $\bar P$ is the minimal solution within the class of functions satisfying the weakly continuity property, when $B$ is a 
 $q$-bounded set, while  \cite[Theorem 4.3]{FMS} claims the minimality within the smaller class of functions
satisfying the weakly continuity property
when  $B\in\B({\bf X}).$
\begin{corollary} \label{CORR2} {\rm (cp.  Feinberg et al.~\cite[Theorem 4.3]{FMS})}
Let Assumption~\ref{LB} hold.
Then $\bar P$ is the minimal function in $\hat{\cal P}$ satisfying  properties (a) and (b) stated in  Corollary~\ref{LEMMAB}.
Furthermore, if   the transition function $\bar{P}$ is regular,
then $\bar{P}$ is the unique element of $\hat{\cal P}$  taking values in $[0,1]$ and satisfying  properties (a) and (b) stated in  Corollary~\ref{LEMMAB}. 
\end{corollary}
\begin{proof}
In view of Lemma~\ref{l:A-eq}(b), the corollary follows from Theorem~\ref{thm:FKDE} and Lemma~\ref{LEMMA5}.
\qed
\end{proof}

The following two corollaries from Corollary~\ref{CORR2} are useful for applying the results of this paper to continuous-time jump Markov decision processes; see Feinberg et al.~\cite[Theorem~3.2]{FMS1}.

\begin{corollary}
\label{Cor-S}
Under Assumption~\ref{LB}, the following statements hold:

(a) for all  $u \in [T_0, T_1[,$ $x \in {\bf X},$ and $q$-bounded sets $B \in \B({\bf X}),$ the function $\bar{P}(u,x;t,B)$ satisfies  the equality in formula~\eqref{eq:FKE} for all $t \in ]u,T_1[.$ 

(b) the function $\bar{P}$ is the minimal function in $\hat{\cal P}$ for which statement~(a) holds. In addition, if the transition function $\bar{P}$ is regular, then $\bar{P}$ is the unique function in $\hat{\cal P}$ with values in $[0,1]$ for which statement~(a) holds.
\end{corollary}
\begin{proof}
 Lemma~\ref{Cor} and Corollary~\ref{LEMMAB} imply that statement~(a) of the corollary holds for a function $P$ from $\hat{ \cal P}$ if and only if the function $P$ satisfies properties (a) and (b) stated in Corollary~\ref{LEMMAB}. Therefore, this corollary follows from Corollary~\ref{CORR2}.
\end{proof}

When $x$ is fixed and $u = T_0$,  formula~\eqref{eq:FKE} is an equation in two variables $t$ and $B$. Hence, for simplicity, we write $P(t,B)$ instead of $P(T_0,x;t,B)$ in \eqref{eq:FKE} for any function $P$ from $\hat{\cal P}$ when $x$ is fixed and $u = T_0,$ and \eqref{eq:FKE} becomes
\begin{equation}
\label{eq:FKE1}
P(t,B) = I\{x \in B\} + \int_{T_0}^{t} ds \int_{\bf X} q(y,s,B \setminus \{y\}) P(s, dy) -\int_{T_0}^{t} ds \int_{B} q(y,s) P(s,dy).
\end{equation}
For fixed $x \in {\bf X}$ and $u = T_0$, the function $\bar{P}(t, \cdot)$ is the marginal probability distribution 
of the  process $\{\BB{X}_t: t \in [T_0, T_1[\}$ at time $t$ given $\BB{X}_{T_0} = x$. Under Assumption~\ref{LB}, the following corollary describes the minimal solution of \eqref{eq:FKE1} and provides a sufficient condition for its uniqueness.
\begin{corollary}
\label{Cor-P}
Fix an arbitrary  $x \in {\bf X}$. Under Assumption~\ref{LB}, the following statements hold:

(a) for all $t \in ]T_0, T_1[$ and $q$-bounded sets $B \in \B({\bf X}),$ the function $\bar{P}(t,B)$ satisfies \eqref{eq:FKE1};

(b)  $\bar{P}(t,B),$ where $t\in ]T_0,T_1[$ and $B\in \B({\bf X}),$ is the minimal non-negative function that is a measure on $({\bf X},\B({\bf X}))$ for fixed $t$, is measurable in $t$ for fixed $B$, and for which statement~(a) holds. In addition, if   the function $q(z,t)$ is bounded on the set ${\bf X} \times [T_0, T_1[$, then $\bar{P}(t,B)$  is the unique non-negative function  with values in $[0,1]$ and satisfying the conditions stated in the first sentence of this statement.
\end{corollary}
\begin{proof}
Statement (a) of the corollary follows immediately from Corollary~\ref{Cor-S}(a) when $u = T_0$. To prove  statement (b),  consider a non-negative function $P(t,B)$, where $t \in ]T_0,T_1[$ and $B \in \B({\bf X}),$ that satisfies the conditions given in the first sentence of statement (b) of this corollary. Define the function $f(u,z;t,B) \in \hat{\cal P}$,
\begin{equation}
\label{sol-alt}
f(u,z ;t, B) =  \left \{ \begin{array}{ll}
P(t, B), & \quad \text{ if} \quad u = T_0 \text{ and } z = x, \\
\bar{P}(u,z;t,B), & \quad \text{ otherwise}.
\end{array}
\right.
\end{equation}
Then, it follows from Corollary~\ref{Cor-S}(a) and \eqref{sol-alt} that the function $f$ satisfies the property given in Corollary~\ref{Cor-S}(a).  Thus,  Corollary~\ref{Cor-S}(b)  and \eqref{sol-alt} imply
\begin{equation}
\label{2-m}
P(t, B) = f(T_0,x;t,B) \ge \bar{P}(T_0,x;t,B) = \bar{P}(t, B), \qquad t \in ]T_0,T_1[, B \in \B({\bf X}).
\end{equation}

To show the uniqueness property, let the function $P$  take values in $[0,1]$. This fact and the property that the function $\bar{P}(u,z;t,B)$ takes values in $[0,1]$ for all $u,$ $z,$ $t,$ $B$ in the domain of $\bar{P}$ imply that the function $f$ defined in \eqref{sol-alt} takes values in $[0,1]$.
 Observe that ${\bf X}$ is a $q$-bounded set if the  function $q(z,t)$ is bounded on the set ${\bf X} \times [T_0, T_1[.$ Then, as follows from Corollary~\ref{Cor-S}(a),  $\bar{P}(u,z;t,{\bf X}) = 1$ for all $u \in [T_0, T_1[,$ $t \in ]u,T_1[,$  and $z \in {\bf X}$. Therefore, it follows from Corollary~\ref{Cor-S}(b) that $f(u,z;t,B) = \bar{P}(u,z;t,B)$ for all $u,z,t,B$ in the domain of $\bar{P}$, which along with \eqref{2-m} implies the uniqueness property of $\bar{P}(t,B).$ \qed
\end{proof}

\begin{acknowledgement}
The first two authors thank Pavlo Kasyanov for useful comments.
\end{acknowledgement}

\end{document}